\documentclass[10pt,reqno]{amsart}
\usepackage{amsmath,amssymb,amsthm,amsfonts}
\usepackage{graphicx}
\usepackage{tikz}

\theoremstyle{plain}
\newtheorem*{lemma}{Lemma}
\newtheorem*{thm}{Theorem}

\begin{document}

\title[]{Rigidity of Kantorovich Solutions\\ in Discrete Optimal Transport}

\author[]{ Alexander Bruce Johnson}
\address{Department of Applied Mathematics, University of Washington, Seattle, USA} 
\email{verano13@uw.edu }

\author[]{Stefan Steinerberger}
\address{Department of Mathematics, University of Washington, Seattle, USA} 
\email{steinerb@uw.edu}

\thanks{\textbf{Note added.} We found out that these ideas are (essentially) contained in \cite[Section 3]{peyre0}. This manuscript is not submitted anywhere, we hope it serves as an interesting exposition. In particular, the main question raised in Section 3 remains open.}
\date{}

\begin{abstract} We study optimal transport plans from $m$ equally weighted points (with weights $1/m$) to $n$ equally weighted points (with weights $1/n$). The Birkhoff-von Neumann Theorem implies that if $m=n$, then the optimal transport plan can be realized by a bijective map: the mass from each $x_i$ is sent to a unique $y_j$. This is impossible when $m \neq n$, however, a certain degree of rigidity prevails. We prove, assuming w.l.o.g. $m < n$, that for generic transport costs the optimal transport plan sends mass from each source $x_i$ to $n/m \leq \mbox{different targets} \leq  n/m + m-1$. Moreover, the average target receives mass from $\leq 1 + m/\sqrt{n}$ sources. Stronger results might be true: in experiments, one observes that each source tends to distribute its mass over roughly $n/m +c$ different targets where $c$ appears to be rather small.
\end{abstract}

\maketitle

\textbf{Note added.} We discovered that a stronger result is known \cite[Proposition 3.4]{cuturi}: it is actually known that there is a transport map with $m+n-1$ nonzero entries. Since, in our setting, each source has to map to at least $\left\lceil n/m \right\rceil$ different targets, we have that, assuming $x_i$ sends mass to $\left\lceil n/m \right\rceil + \varepsilon_i$ targets, that 
$$ n + \sum_{i=1}^{m} \varepsilon_i \leq  \sum_{i=1}^{m} \left\lceil n/m \right\rceil + \varepsilon_i \leq n+m - 1$$
which implies our result and also explains why we typically observe $\left\lceil n/m \right\rceil + c$ (with $c$ small). We hope the reader may nonetheless enjoy some of the illustrations.

\section{Introduction}
\subsection{Discrete Optimal Transport} The goal of this paper is to introduce and discuss a rigidity phenomenon in discrete Optimal Transport. We assume that $X = \left\{x_1, \dots, x_m\right\}$ is a set of sources that are all assumed to have the same weight which is normalized to be $1/m$. This set is matched to $n$ targets $Y = \left\{y_1, \dots, y_n\right\}$ which all have the same `capacity' normalized to be $1/n$. The cost of transporting mass from $x_i$ to $y_j$ is given by $c(x_i, y_j)$ and these transport costs encode the entire underlying geometry. We are interested in the optimal transport solution sending mass from the sources to the targets. Since everything is discrete, one can rewrite this problem in terms of matrices. Using $P_{ij}$ to denote the amount of mass sent from $x_i$ to $y_j$, we are interested in the  matrix $P \in \mathbb{R}_{\geq 0}^{m \times n}$ solving
$$ \sum_{i=1}^{m} \sum_{j=1}^{n} c(x_i, y_j) \cdot P_{ij} \rightarrow \min \qquad \mbox{subject to} \qquad \sum_{j=1}^{n} P_{ij} = \frac{1}{m}, \sum_{i=1}^{m} P_{ij} = \frac{1}{n}.$$
The restrictions ensure that each source sends all its measure to some targets and that each target receives mass corresponding exactly to its capacity.
In the particularly interesting case of trying to minimize the Wasserstein $W_p$ distance, we would have $c(x_i, y_j) = \|x_i - y_j\|^p$.
 A crucial aspect of the problem is that mass can be split: the mass located at a source may be split into different pieces each of which gets sent to a different target. This is the Kantorovich formulation of Optimal Transport.
Summarizing, we are moving mass from each source point so that every target is filled to capacity in such a way that the overall transport cost is minimized; mass is allowed to be split and sent to different targets.

\subsection{Birkhoff-von Neumann}
A classical result states that bistochastic matrices are the convex hull of permutation matrices. This has the following beautiful implication for discrete Optimal Transport (see \cite{birk, brezis, konig, merigot, peyre, von}). 

\begin{thm}[Birkhoff-von Neumann] When $m=n$, then there exists an optimal transport plan that is a \textit{bijective} map between $X$ and $Y$.
\end{thm}

When the number of sources and targets coincide, then the most efficient way of transporting mass does not involve mass splitting: the mass from each source is sent to exactly one target. The Theorem does not require any assumptions about the cost of transporting mass from $x_i$ to $y_j$. It has a clear conceptual proof. When $m=n$, the matrices in our problem are bistochastic matrices. We are thus dealing with an affine cost function over bistochastic matrices. A theorem of Birkhoff \cite{birk} (also attributed to K\"onig \cite{konig} and von Neumann \cite{von}) says that the bistochastic matrices are the convex hull of the permutation matrices. The minimum of an affine cost function in a non-empty polyhedron is attained in an extremal point (a permutation matrix) and this implies the result independently of transport costs.
The solution need not be unique, non-bijective optimal transport plans can exist.

\subsection{The case $m\neq n$} It is clear that this entire argument breaks down when $m \neq n$. However, in some cases, as was observed in \cite{bamdad}, some regularity prevails.

\begin{thm}[Hosseini-Steinerberger]
  There exists an optimal transport plan such that sources are moved to at most $n/\mathrm{gcd}(m,n)$ different targets and each target receives mass from at most $m/\mathrm{gcd}(m,n)$ different sources.
\end{thm}

This follows from artifically replacing points by multiple `dummy' points all located in the same spot and applying Birkhoff-von Neumann. For example (see Fig. \ref{fig:one}), when sending mass from $m=20$ sources to $n=30$ targets, the result guarantees that this can be done in a way that sends mass from each source to at most 3 different targets and where each target receives mass from at most 2 different sources. The result does not say anything when $m$ and $n$ are coprime.
\begin{center}
    \begin{figure}[h!]
    \begin{tikzpicture}
\node at (0,0) {\includegraphics[width=0.35\textwidth]{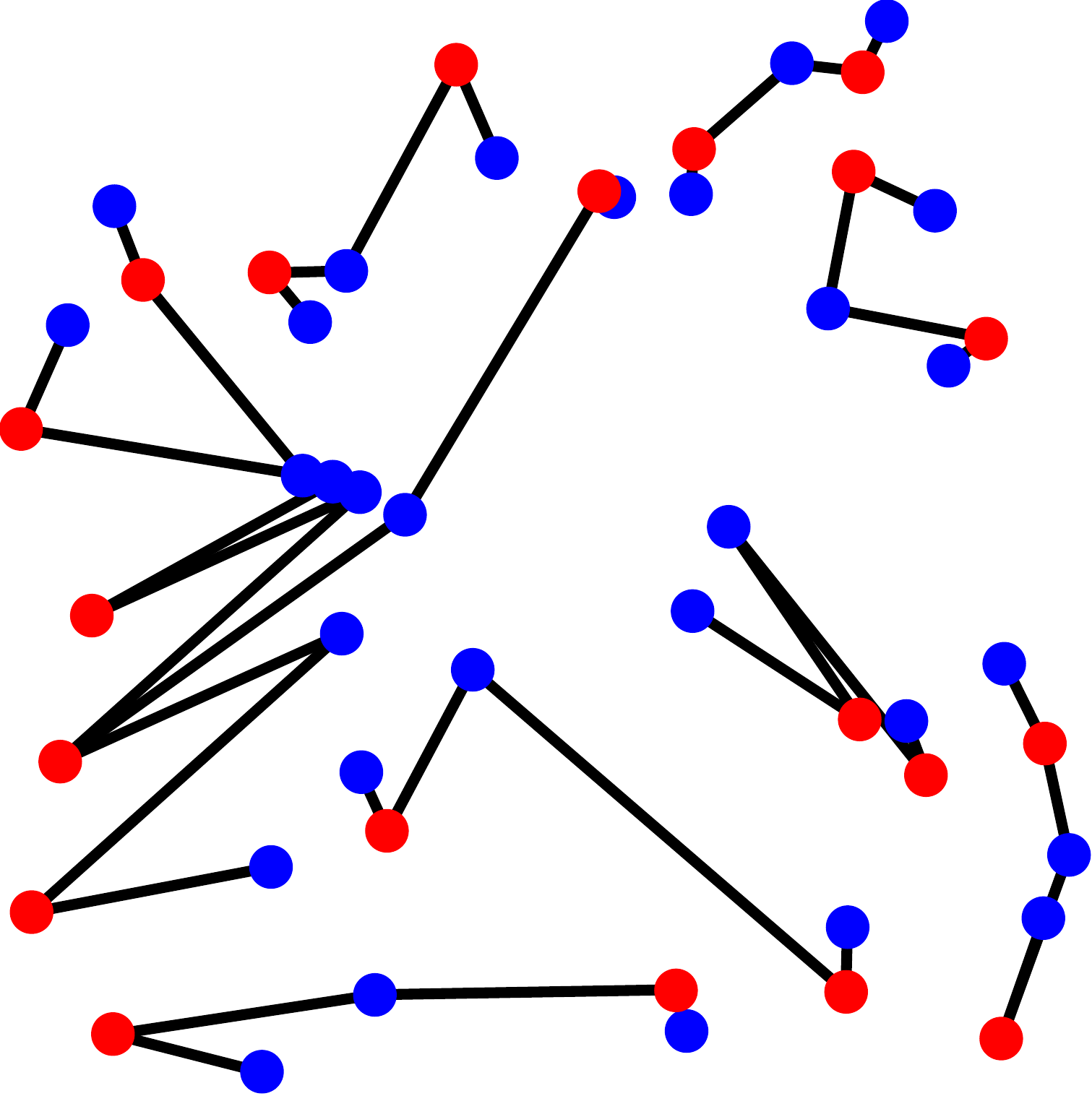}};
\node at (6,0) {\includegraphics[width=0.35\textwidth]{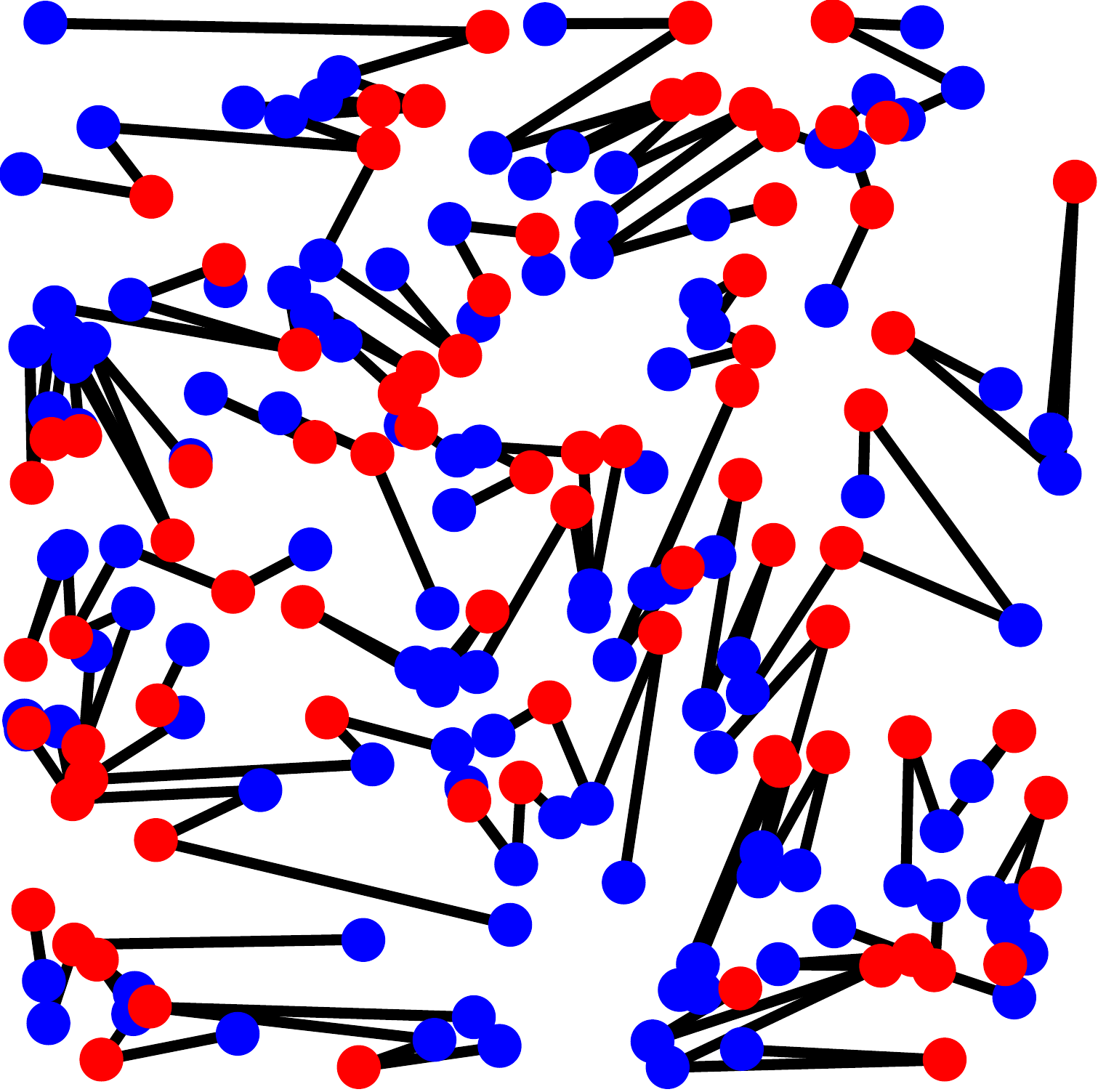}};
    \end{tikzpicture}
    \caption{$m=2\ell$ source points (red) are mapped to $n = 3\ell$ targets (blue) to minimize $W_1$ (left: $\ell=10$, right: $\ell=40$). Independent of $\ell$, each source sends mass to at most 3 different targets, each target receives mass from at most 2 different sources. }
    \label{fig:one}
    \end{figure}
\end{center}

\section{Kantorovich Rigidity}
\subsection{Rigidity.} The optimal transport plan is symmetric: one may think of sending mass from a probability measure $\mu$ to a probability measure $\nu$ or, equivalently, mass from $\nu$ to $\mu$. We can thus assume, without loss of generality, that $m \leq n$. In this setting, every source has a `large' weight $1/m$ while every target has `smaller' weight $1/n$. It follows that each source has to send mass to multiple different targets and it's easy to see that at least $\left\lceil n/m\right\rceil$ are needed since each source must deplete all of its mass. It appears as if the optimal transport plan will send mass to a number of targets that is very rigid and close to this trivial lower bound of $\left\lceil n/m\right\rceil$.
 \begin{center}
    \begin{figure}[h!]
    \begin{tikzpicture}
\node at (0,0) {\includegraphics[width=0.45\textwidth]{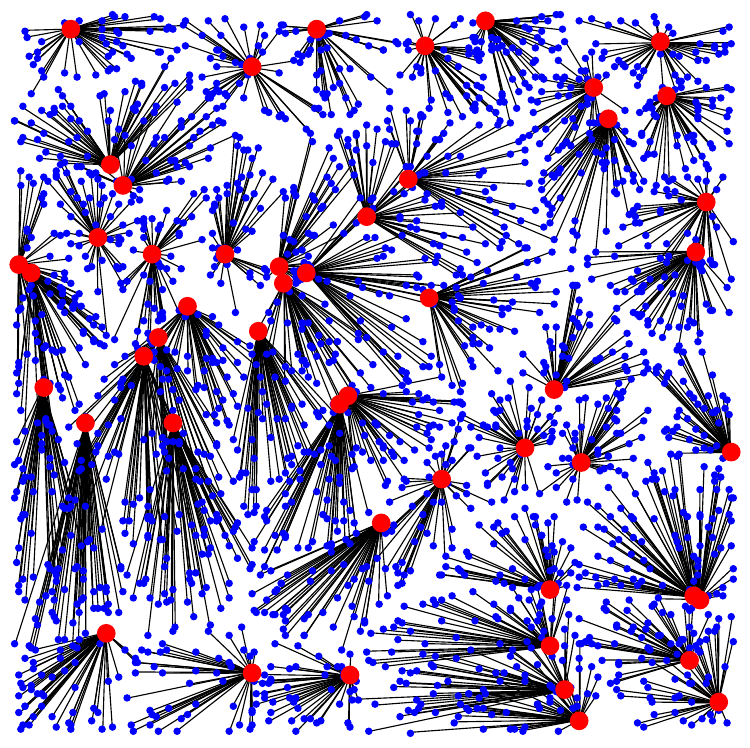}};
\node at (6,0) {\includegraphics[width=0.45\textwidth]{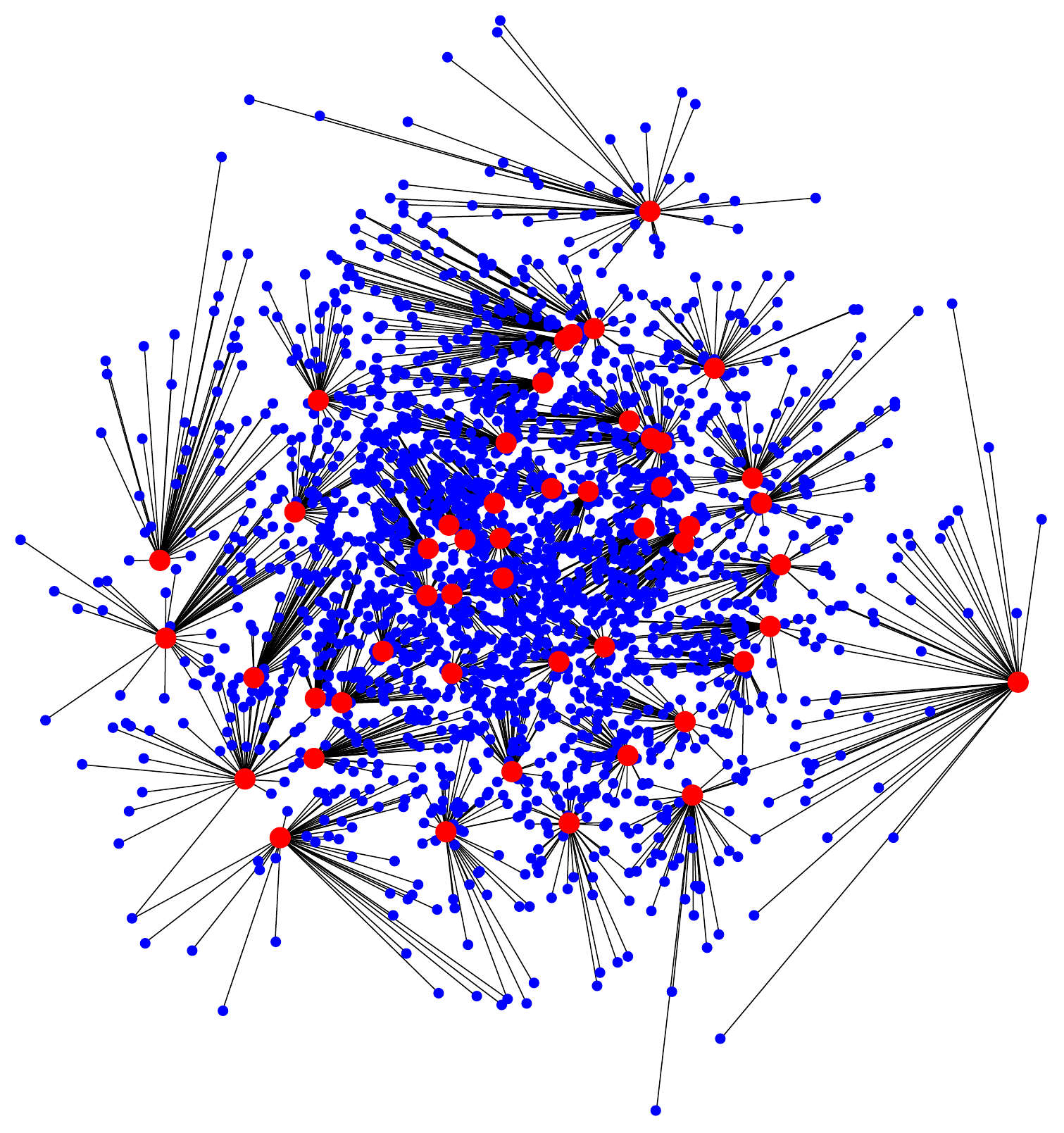}};
    \end{tikzpicture}
    \caption{50 sources are mapped to $2222$ targets (left: iid samples in $[0,1]^2$, right: Gaussian) to minimize the $W_2$ cost. Each source is sent to either $\left\lceil 2222/50\right\rceil = 45$, 46 or 47 targets. }
    \label{fig:two}
    \end{figure}
\end{center}

One such example is shown in Fig. \ref{fig:two}. The phenomenon appears to be generic, it is easy to come up with more examples. One might be tempted to ask whether there might not be a general principle at play. This appears to be a fascinating problem.
\begin{quote}
    \textbf{Question.} For equally weighted points and $m \leq n$, does the optimal transport plan always send mass from each source to a number of targets that is either $\left\lceil n/m\right\rceil$ or just a little bit larger than that?
\end{quote}
This appears to be the case when doing numerical experiments and it seems to be at least `generic' in the sense that counterexamples (if any) appear to be rare.

\subsection{Main Result.} We can now describe our main result. It was discovered empirically when looking at transport plans from $m=7$ sources to $n=2000$ targets: we found that the optimal transport plan tends to send mass from each source to somewhere between $\left\lceil n/m\right\rceil = 286$ and 289 different targets. This remained true even after replacing the transport costs $c(x_i, y_j)$ by random numbers.
The trivial lower bound of 286 is not surprising; we were somewhat surprised by how accurately it seems to predict the true behavior. This led to the following result.

\begin{thm} If the transport costs $c(x_i, y_j)$ from $x_i$ to $y_j$ satisfy
$$\forall~i \neq j, k \neq \ell \qquad \quad c(x_i, y_k) + c(x_j, y_{\ell}) \neq c(x_i, y_{\ell}) + c(y_j, x_k),$$
and $m \leq n$, then every optimal transport plan satisfies, for all $1 \leq i\leq m$,
    \begin{equation}
      \left\lceil \frac{n}{m} \right\rceil \leq  \# \left\{1 \leq j \leq n: x_i~\mbox{sends mass to}~y_j \right\}  \leq \left\lfloor \frac{n}{m}\right\rfloor + m -1.      
    \end{equation}
    Moreover, the average source sends mass to relatively few targets
   \begin{equation}
       \left\lceil \frac{n}{m} \right\rceil \leq \frac{1}{m} \sum_{i=1}^{m} \# \left\{1 \leq j \leq n: x_i~\mbox{sends mass to}~y_j \right\} \leq \frac{n}{m} + \sqrt{n}.    
   \end{equation}
   and the average target receives mass from relatively few sources
   \begin{equation}
        \frac{1}{n} \sum_{j=1}^{n} \# \left\{1 \leq i \leq m: x_i~\mbox{sends mass to}~y_j \right\} \leq 1+ \frac{m}{\sqrt{n}}  
   \end{equation}
\end{thm}
\textit{Remarks.}
\begin{enumerate}
    \item The assumption on the transport costs is generically satisfied. If the assumption is not satisfied then arbitrarily small random perturbations of the transport costs is guaranteed to then satisfy it almost surely.
    \item  The lower bound $\left\lceil n/m \right\rceil$ in (1) is trivially optimal. 
    \item The upper bound in (1) implies (2) when $m \leq \sqrt{n}$.
    \item Inequality (3) proves generic rigidity for targets. When $n \gg m^2$, then most targets get their mass from a single source.
\end{enumerate}

It is not clear to us how far these upper bounds are from the truth. We observe in examples that the upper bound $m-1$ in (1) can be replaced by something that is either $\mathcal{O}(1)$ or at least seems to grow rather slowly in $m$ and $n$. It would be of great interest to understand this better. It is conceivable that our results are optimal but the situation in which they are sharp are rather rare -- in that case it would be interesting to understand what the behavior would be like in the `generic' setting for a suitable definition of `generic'. It is also quite possible that much stronger results are true. Finally, we note that it could be interesting to understand the case where both sources and targets are i.i.d. samples from the same probability distribution and whether tools from probability theory allow for a refined analysis (tight concentration, a Central limit Thorem, ...) of that case.

\section{Proof of the Theorem}

\subsection{A Non-Crossing Lemma}
The purpose of this section is to prove a simple non-crossing Lemma that seems like it could have already appeared in the Literature, possibly many times, but we were unable to locate it.

\begin{center}
    \begin{figure}[h!]
\begin{tikzpicture}[scale=1.3]
    \filldraw (0,0) circle (0.07cm);
        \filldraw (0,1) circle (0.07cm);
      \filldraw (2,0) circle (0.07cm);
        \filldraw (2,1) circle (0.07cm);  
\draw [->, thick] (0,0) -- (1.8, 0);
  \draw [->, thick] (0,0) -- (1.8, 0.9);      
 \draw [->, thick] (0,1) -- (1.8, 0.1);
  \draw [->, thick] (0,1) -- (1.8, 1);   
  \node at (1, 1.15) {\Large $a$};
    \node at (0.75, 0.8) {\Large $b$};
        \node at (0.3, 0.32) {\Large $c$};
    \node at (1, -0.15) {\Large $d$};
    \node at (-0.3, 0) {$p_2$};
       \node at (-0.3, 1) {$p_1$};
         \node at (2.3, 0) {$p_3$};
       \node at (2.3, 1) {$p_4$};
\end{tikzpicture}
\caption{This kind of transport (sub-)plan can never appear.}
\label{fig:forbidden}
    \end{figure}
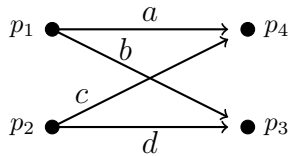
\end{center}

\begin{lemma}
    Suppose we have $2$ sources and $2$ targets with transport costs a, b, c, and d as shown in Fig. \ref{fig:forbidden}. If $a+d \neq b+c$, then the optimal solution does not transport mass from each source to each target.
\end{lemma}
\begin{proof} Suppose we are given an optimal transport plan which contains such a configuration. We denote the amount of mass being transported from $p_i$ (where $i \in \left\{1,2\right\}$ to $p_j$ (where $j \in \left\{3,4\right\}$) by $m_{ij}$. Assume all 4 numbers are positive and let 
$$ \varepsilon = \frac{1}{2} \min\left\{ m_{13}, m_{14}, m_{23}, m_{24} \right\}.$$
Consider the following action: decrease $m_{13}$ by $\varepsilon$ and increase $m_{14}$ by $\varepsilon$ while simultaneously decreasing $m_{24}$ by $\varepsilon$ and increasing $m_{23}$ by $\varepsilon$. This remains a valid transport plan insofar as all sources remain depleted and all targets remain filled at a capacity. The change in transport cost is given by
$$ -b \varepsilon + a\varepsilon - c\varepsilon + d \varepsilon = \varepsilon(a - b - c + d).$$
We also consider a different action: decrease $m_{14}$ by $\varepsilon$ and increase $m_{13}$ by $\varepsilon$ while simultaneously decreasing $m_{23}$ by $\varepsilon$ and increasing $m_{24}$ by $\varepsilon$. This leads to a change in transport cost
$$ -a \varepsilon + b\varepsilon - d\varepsilon + c \varepsilon = \varepsilon(-a + b + c -d).$$
We note that this has the exact opposite sign of the first action. Thus, if
$$ a + d \neq b + c,$$
then both these actions have nonzero effect and, since they are of opposite sign, one of them will lead to an improvement. This contradicts the assumption that the transport plan was optimal and proves the Lemma.
\end{proof}

\subsection{Proof of the Theorem}
\begin{proof}
There are $m$ sources and $n$ targets and $m \leq n$.  We introduce, for $1 \leq i \leq m$, the number $t_i$ of targets that $x_i$ sends mass to:   
$$ t_i = \# \left\{1 \leq j \leq n: x_i~\mbox{sends mass to}~y_j \right\}.$$
It is clear that each source has to send mass to at least $\left\lceil n/m \right\rceil$ different targets which provides a lower bound on $t_i$.
As for the upper bound, we partition the points that $x_i$ sends mass to into two sets, $A$ and $B$: those that get filled to capacity by $x_i$,
$$ A = \left\{1 \leq j \leq n: x_i~\mbox{sends}~\frac{1}{n}~\mbox{mass to}~ y_j \right\},$$
and the remaining points in the set $B$. We have $t_i = \# A + \# B$. By mass considerations, we have $\# A \leq \left\lfloor n/m \right\rfloor$. It remains to bound the number of points in $B$. Any such point receives less than $1/n$ mass from $x_i$ which means it has to receive additional mass from other sources. If $y_s, y_t \in B$ are two distinct targets in $B$, then they both receive mass from $x_i$ and hence cannot both receive mass from $x_{\ell}$ where $\ell \neq i$ because that is forbidden by the Lemma: two targets ($y_s$ and $y_t$) simultaneously receiving mass from two sources ($x_i$ and $x_{\ell}$). This means that we can map each $y_s \in B$ injectively to a non-empty subset of $\left\{x_1, \dots, x_{i-1}, x_{i+1}, x_m\right\}$ (this subset being the other sources that $y_s$ receives mass from) and thus $\# B \leq m-1.$ 
Altogether we deduce
$$ t_i = \# A + \#B \leq \left\lfloor\frac{n}{m} \right\rfloor + m - 1.$$
This proves the first part of the statement. It will be convenient to prove the third part next. For the third part, we introduce, for $1 \leq j \leq m$, the number of $\ell_j$ of different sources sending mass to $y_j$
$$ \ell_j = \# \left\{1 \leq i \leq m: x_i~\mbox{sends mass to}~y_j \right\}.$$
We have $\ell_j \geq 1$ since every target receives mass from somewhere. We now use the non-crossing Lemma once more as follows: a pair of sources $\left\{x_a, x_b\right\}$ can only send mass to at most one single target; if they were to send mass simultaneously to two different targets, this would violate the non-crossing Lemma. If a target $y_j$ receives mass from $\ell_j \geq 2$ sources, then we can find $\binom{\ell_j}{2}$ different pairs that send mass to $y_j$. Each pair appears at most once and there are $\binom{m}{2}$ pairs in total, thus
$$ \sum_{j=1}^{n} \binom{\ell_j}{2} \leq \binom{m}{2},$$
where we use the convention $\binom{1}{2} = 0$.
The Cauchy-Schwarz inequality implies
\begin{align*}
    \sum_{j=1}^{n} \ell_j &= n +  \sum_{j=1}^{n} (\ell_j - 1) \leq n + \sqrt{n} \left( \sum_{j=1}^{n} (\ell_j - 1)^2 \right)^{1/2} \\
    &\leq n + \sqrt{n} \left( \sum_{j=1}^{n}  2\cdot  \binom{\ell_j}{2} \right)^{1/2} \leq n + \sqrt{n} \cdot m
\end{align*}
implying the third statement. 
This then implies the second statement since
\begin{align*}
\sum_{j=1}^{n}  \ell_j &= \sum_{j=1}^{n} \# \left\{1 \leq i \leq m: x_i~\mbox{sends mass to}~y_j \right\}\\
&= \sum_{j=1}^{m}\# \left\{1 \leq j \leq n: x_i~\mbox{sends mass to}~y_j \right\}.
\end{align*}
\end{proof}


\vspace{-15pt}

\begin{thebibliography}{10}

\bibitem{birk} G. Birkhoff. Tres observaciones sobre el algebra lineal. Universidad Nacional de Tucum\'an
Revista Series A, 5:147--151, 1946.

\bibitem{brezis} H. Brezis, Remarks on the Monge-Kantorovich problem in the discrete setting. 
C. R. Math. Acad. Sci. Paris 356 (2018), no. 2, 207--213.

\bibitem{cuturi} Gabriel Peyr\'e, Marco Cuturi, Computational Optimal Transport. arXiv:1803.00567

\bibitem{bamdad} B. Hosseini and S. Steinerberger, Intrinsic Sparsity of Kantorovich solutions, Comptes Rendus Math\'{e}matique 360 (2022), p. 1173-1175.

\bibitem{konig} D. K\"onig, Theorie der endlichen und unendlichen Graphen, Akademische. Verlags Gesellschaft,
Leipzig, 1936.

\bibitem{merigot} Q. Merigot and B. Thibert. Optimal transport: discretization and algorithms. In Handbook of Numerical Analysis, vol. 22, pp. 133-212. Elsevier, 2021.

\bibitem{peyre0} G. Peyr'{e}, Course Notes on Computational Optimal Transport, 2021

\bibitem{peyre} G. Peyr\'e and M. Cuturi,  Computational optimal transport: With applications to data science. Foundations and Trends in Machine Learning, 11(5-6), 355--607.

\bibitem{von} J. von Neumann,  A certain zero-sum two-person game equivalent to an optimal assignment
problem, Ann. Math. Studies 28:5--12, 1953.


\end{thebibliography}
\end{document}